\documentclass[11pt,a4paper]{article}
\usepackage[english]{babel}
\usepackage[T1]{fontenc}
\usepackage{ae,aecompl}
\usepackage[utf8]{inputenc} 
\usepackage{amsmath,amsfonts,amsthm}
\usepackage{cite} 
\usepackage[linktocpage]{hyperref} 
\usepackage{graphicx} 
\usepackage{float}
\usepackage{color} 
\usepackage{pdfpages}
\usepackage{fullpage}

\numberwithin{equation}{section} 

\newtheorem{thm}{Theorem}[section]
\newtheorem{cor}[thm]{Corollary}
\newtheorem{prop}[thm]{Proposition}
\newtheorem{lem}[thm]{Lemma}
\newtheorem{rmq}[thm]{Remark}

\newtheorem{deff}[thm]{Definition}

\newcommand{\ti}{\tilde}

\DeclareMathOperator{\sgn}{sgn}
\DeclareMathOperator{\diverg}{div}

\DeclareMathOperator*{\argmin}{\arg\!\min}

\def\E{\mathcal{E}}

\def\R{\mathbb{R}}

\def\N{\mathbb{N}}

\def\D{\mathbb{D}}

\def\e{\varepsilon}

\def\O{\Omega}
\def\BV{BV(\Omega)}
\def\la{\lambda}
\def\vphi{\varphi}
\def\G{\nabla}

\title{Some remarks on the staircasing phenomenon in total-variation based
image denoising      
}

\author{Khalid Jalalzai\thanks{CMAP, CNRS, Ecole Polytechnique, 91128 Palaiseau Cedex, France ({\tt khalid.jalalzai@polytechnique.edu}).}
}

\begin{document}
\date{April 10, 2012}
\maketitle

\begin{abstract}
This paper deals with the so-called staircasing phenomenon, which frequently arises in total variation based denoising models in image analysis. We prove in particular that staircasing always occurs at global extrema of the datum and at all extrema of the minimizer. It is also shown that, for radial images, the staircasing always appears at the extrema and at the boundary of the image. We also prove the equivalence between the denoising model and the total variation flow, in the radial case, thus extending a previous result in dimension one. This equivalence cannot hold in the nonradial case, as it is shown with a counterexample. This connection allows us to understand how the staircase zones and the discontinuities of the denoising problem evolve with the regularization parameter.
\end{abstract}


\textbf{AMS:} 35J70, 65J20, 35K65, 68U10.\\

\textbf{Acknowledgment:} I warmly thank Antonin Chambolle who suggested me this fruitful research project.

\thispagestyle{plain}
\markboth{Khalid Jalalzai}{Some remarks on the staircasing phenomenon in TV denoising}

\section{Introduction}
Functions of bounded variation equipped with the total variation semi-norm were introduced for image reconstruction in 1992. Since then, they have had many successful applications for inverse problems in imaging. Indeed, the penalization of the total variation has the ability to smooth out the image by creating large regular zones and to keep the edges of the most important objects in the image. In this paper we aim to study the latter property in the continuous setting.\\

We assume that a corrupted image $g:\O\subset\R^2\to\R$ went through a degradation
\begin{align*}
g=g_0+n
\end{align*}
where $g_0$ is the original clean image, $n$ is a Gaussian white noise of standard deviation $\sigma$. Rudin, Osher and Fatemi (ROF) proposed in \cite{Rudin92} to minimize the total variation
\begin{align*}
u\mapsto TV(u)=\int_\O|Du|
\end{align*}
amongst functions of bounded variation under the constraint $\|u-g\|_2^2\leq\sigma^2|\O|^2$ to solve the inverse problem and thus get a restored image $u$.
It was proven in \cite{ChamLions} that one can solve in an equivalent way the unconstrained problem
\begin{align*}
\min_{u\in\BV}\la\int_\O|Du|+\frac{1}{2}\|u-g\|_2^2
\end{align*}
for an adequate Lagrange multiplier $\la$. In the literature the minimization of ROF's energy is referred to as the \textit{denoising problem}.

Another possibility is to consider the \textit{total variation flow} for restoring $g$.  A denoised image is given by $u(t)$ that solves
\begin{align*}
\begin{cases}
-\partial_t u(t)\in\partial TV(u(t)) \text{ a.e. }t\in[0,T],\\
u(0)=g.
\end{cases}
\end{align*}
which has a unique solution according to \cite{Andreu}. As we shall see it is not true in general that these approaches coincide.

It has been long observed that using the total variation has the advantage of recovering the discontinuities quite well. In \cite{ChamJump}, the authors carried out a study of the behavior of the minimizer of the denoising problem at these discontinuities. Their results are extended in \cite{JalalzaiJump,JalalzaiPhD} where it is proven that the discontinuities of the minimizer of the anisotropic total variation are contained in those of the datum $g$. Moreover, it is established that one can observe new discontinuities in the weighted case if the weight is merely Lipschitz continuous.

In the present paper, we shall focus on another very important property of the total variation: it smoothes the highly oscillating regions by creating large constant zones which is known in the literature as the \textit{staircasing effect}. This phenomenon is sometimes not desirable in imaging applications since it yields blocky and non-natural structures. It was already studied in \cite{Ring} for the one-dimensional case. Actually, the author proves that whenever the data $g\not\in BV(a,b)$, the minimizer $u'_\lambda$ vanishes almost everywhere.
In \cite{Nikolova00}, Nikolova proves that the staircasing effect is related to the non-differentiability of the total variation term. More precisely, large homogeneous zones are recovered from noisy data and remain unchanged for small perturbations. In other words, the creation of such zones is quite probable. On the contrary, absence of staircasing with the differentiable approximation $TV_\e=\sqrt{|Du|^2+\e}$ is also proven. The latter approaches were carried out only for finite dimensional approximations of the total variation. In \cite{Louchet}, Louchet and Moisan proposed an alternative to the minimization of the total variation by considering the TV-LSE filter (see \cite{Jalalzai} for another alternative). In \cite{Louchet}, the authors proved in the discrete setting that TV-LSE avoids the staircasing effect (in the sense that a region made of 2 pixels or more where the restored image is constant almost never occurs).
As far as we know, there is no result on the subject in the higher dimensional and continuous setting for the classical total variation functional. 
We shall show that staircasing always occurs (even without addition of noise) at global extrema of the datum and at all extrema of the minimizer.
	
In the last section, we investigate further these qualitative properties. An interesting question is to understand how the staircase zones and the discontinuities evolve with the regularization parameter $\la$. The idea is to use the results that are already established for the flow. Unfortunately, in higher dimension the connection between the flow and ROF's energy fails (we give a counterexample). 
However, we are going to prove that this connection actually holds for radial functions. This way, one can prove that the discontinuities form a decreasing sequence, whereas the staircase zones increase with the regularization parameter (which is not true in general). 

Let us remark that all the results are established in dimension $N\geq 2$ since the situation is quite well understood in the one-dimensional case and was widely studied in the literature (see the recent paper \cite{Bonforte} for instance). Indeed, in dimension one, ROF's denoising problem reads as follows
\begin{align}\label{chapTV:ROF_1_D}
\min_{u\in BV(\R)} \int \la|u'(x)|+\frac{1}{2}{(u-g)}^2(x)dx
\end{align}
for $g\in L^2(\R)$ and some positive real $\la$. Let us denote $u_\la$ the minimizer of this problem.

Writing down the Euler-Lagrange (see equation \ref{chapTV:ROF_EL}) one immediately sees
that either $u_\la$ is constant or $z_\la=\sgn(u_\la')$ and as a consequence $u_\la=g$. This is an almost explicit formulation of the solution that tells us that
\begin{itemize}
\item[-] the discontinuities of $u_\la$ are contained in those of $g$,
\item[-] flat zones are created at maxima and minima of $g$.
\end{itemize}
This can be seen in the following simulation:
\begin{figure}[H]   \begin{minipage}[c]{\linewidth}
     \includegraphics[width=\hsize,height=0.1\vsize]{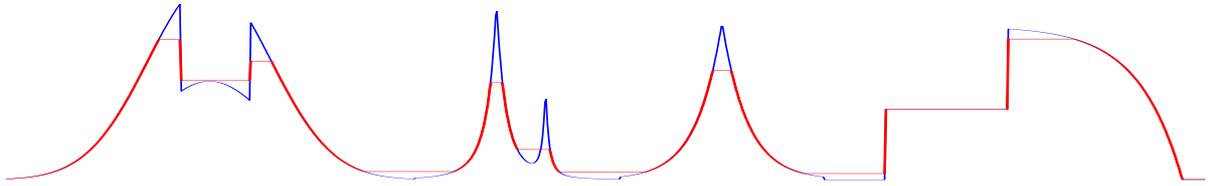}
    \caption
{Minimizer $u_\la$ (in red) of a 1D data $g$ (in blue).}
  \end{minipage}
\end{figure}

The other important result in dimension one is the link with the flow of the total variation. It is known that problem (\ref{chapTV:ROF_1_D}) with $\la=t$ is minimized by $u(t)$, the unique solution of the flow. In the recent article \cite{Briani}, the authors used this relation to prove that for a signal that went through an addition of noise (that is the trajectory of a Wiener process strictly speaking), staircasing occurs almost everywhere. This observation seems to be more general as can be seen in the following test:
\begin{figure}[H]   \begin{minipage}[c]{\linewidth}
     \includegraphics[width=\hsize,height=0.1\vsize]{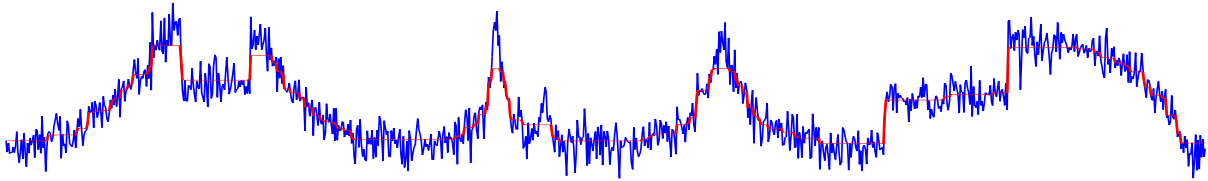}
    \caption
{Minimizer $u_\la$ (in red) of a noisy data $g$ (in blue) is constant almost everywhere.}
  \end{minipage}
\end{figure}

\section{Mathematical preliminary}
Henceforth $\O$ will denote an open subset of $\R^N$ with Lipschitz continuous boundary. The material of this section can be found in the classical textbooks~\cite{Ambrosio00,Giusti,Ziemer} but also in the recent survey~\cite{ChamTV}. 
\subsection{Functions of bounded variation}Let us start with following fundamental definition:
\begin{deff}
A function $u\in L^1(\O)$ is of \textit{bounded variation} in $\O$ (denoted $u\in BV(\O)$) if its distributional derivative $Du$ is a vector-valued Radon measure that has finite total variation \textit{i.e.} $|Du|(\O)<\infty$.
By the Riesz representation theorem, this is equivalent to say that
$$|Du|(\Omega)=\sup\left\{\int_\O u\diverg\vphi\ /\ \vphi\in C_c^\infty(\O,\R^N), \forall x\in\O\ |\vphi(x)|\leq 1\right\}<\infty.$$
In the sequel, the quantity $|Du|(\O)$ also denoted $\int_\O|Du|$ or simply $TV(u)$ will be called the \textit{total variation} of $u$. It is readily checked that ${\|\cdot\|}_1+TV$ defines a norm on $\BV$ that makes it a Banach space.\newline
\end{deff}

A first result that is a straightforward consequence of the dual definition we just gave is a key step to apply the direct method in the calculus of variations:
\newline
\begin{prop}[Sequential lower semicontinuity]\label{chapTV:lsc}
Let ${(u_n)}_{n\in\N}$ be any \linebreak sequence in $BV(\O)$ such that $u_n\to u$ in $L^1(\O)$ then
\begin{align*}
\int_\O|Du|\leq\liminf_{n\to\infty} \int_\O|Du_n|.
\end{align*}
\end{prop}

One also has
\begin{prop}[Approximation by smooth functions]\label{chapTV:approxBV}
If $u\in\BV$ then there exists a sequence ${(u_n)}_{n\in\N}$ of functions in $C^\infty(\O)$ such that
\begin{align*}
u_n\to u\text{ in } L^1(\O)\\
\int_\O|\nabla u_n| \to \int_\O |Du|.
\end{align*}
\end{prop}

For the direct method to apply it is usually of importance to have a compactness result:
\begin{thm}[Rellich's compactness in $BV$]\label{chapTV:Rellich}
Given a bounded $\O\subset\R^N$ with Lipschitz boundary and any sequence ${(u_n)}_{n\in\N}$ such that \linebreak $\left({\|u_n\|}_{L^1(\O)}+\int_\O |Du_n|\right)$ is bounded,
there exists a subsequence ${(u_{n(k)})}_{k\in\N}$ that converges in $L^1$ to some $u\in\BV$ as $k\to\infty$.
\end{thm}

\begin{deff}
Let $E\subset\R^N$ be a Borelian set. It is called a \textit{set of finite perimeter} or also \textit{Cacciopoli set} if $u=\chi_E$ is a function of bounded variation. We will call \textit{perimeter} of $E$ in $\O$, and denote $P(E,\O)$ or simply $P(E)$ (if $\O=\R^N$), its total variation.
\end{deff}

The following key result provides a connection between the total variation of a function and the perimeter of its level sets.
\newline
\begin{thm}[Coarea formula]\label{chapTV:CoareaAFP}
If $u\in\BV$, the set $E_t=\{u>t\}$ has finite perimeter for a.e. $t\in\R$ and
$$|Du|(B)=\int_{-\infty}^\infty|D\chi_{\{u>t\}}|(B)dt$$
for any Borel set $B\subset\O$.
\end{thm}

Finally let us recall the following
\newline
\begin{thm}[Sobolev inequalities]\label{chapTV:Sobo}
Let $u\in\BV$ and let us denote $\langle u \rangle=\frac{1}{|\O|}\int_\O u$.
For a bounded Lipschitz domain $\O$, the following Poincaré inequality holds 
\begin{align*}
\|u-\langle u \rangle\|_{L^{\frac{N}{N-1}}(\O)}\leq C(N,\O)\int_\O |Du|.
\end{align*}
If $\O=\R^N$, one has the Sobolev inequality
\begin{align*}
\|u\|_{L^{\frac{N}{N-1}}(\R^N)}\leq C(N)\int_\O |Du|.
\end{align*}
In particular, if $u=\chi_E$, one gets the isoperimetric inequality
\begin{align*}
|E|^{\frac{N-1}{N}}\leq C(N)P(E,\O).
\end{align*}
\end{thm}

\subsection{BV functions in image processing}

The classical model of a functional where total variation plays a key role is the so-called Rudin-Osher-Fatemi energy:
\begin{align}\label{chapTV:ROF}
\tag{ROF}
\E_\lambda(u)=\lambda\int_\O|Du|+\frac{1}{2}{\|u-g\|}^2_2
\end{align}
In the sequel we shall be interested in minimizing this energy in $BV(\O)$ for some positive real $\lambda$. By Proposition \ref{chapTV:lsc}, there is a minimizer in $BV(\O)$, denoted $u_\lambda$ in the sequel (uniqueness follows from the strict convexity of the energy).

The energy $\E_\la$ is not smooth though convex so it is still possible to get the Euler-Lagrange equation as follows:
\newline
\begin{prop}
Function $u_\la$ minimizes $\E_\la$ in $BV(\O)$ if and only if there exists $z_\la\in L^{\infty}(\O,\R^N)$ such that
\begin{align}\label{chapTV:ROF_EL}
\begin{cases}
-\la\diverg z_\la+u_\la=g & \text{ in }\O,\\
|z_\la|\leq 1 & \text{ in }\O,\\
z_\la\cdot Du_\la=|Du_\la|,\\
z_\la\cdot\nu=0 & \text{ on }\partial\O,\\
\end{cases}
\end{align}
with $\nu$ denoting the inner normal to $\O$.
\end{prop}

In the sequel, we shall only consider Neumann boundary conditions but we could of course take into account Dirichlet or more complicated conditions.
\newline
\begin{rmq}
$z_\la\cdot Du_\la$ is the pairing of a bounded function with a bounded measure and should be understood in the sense of Anzelotti \cite{Anzelotti}.
\end{rmq}

By the coarea formula, the superlevel sets $\{u_\lambda>t\}$ are sets of finite perimeter for almost every $t$ that satisfy the following minimal surface problem:
\newline
\begin{thm}\label{chapTV:LevelSetPb}
Let $u_\la$ be the minimizer of (\ref{chapTV:ROF}). Then for any $t\in\R$, $\{u_\la>t\}$ (resp. $\{u_\la\geq t\}$) is the minimal (resp. maximal) solution of the minimal surface problem
\begin{align}\label{chapTV:LevelSetPbEq}
\min_E \la P(E,\O)+\int_E \left(t-g(x)\right)dx
\end{align}
over all sets of finite perimeter in $\O$. Moreover $\{u_\la>t\}$ being defined up to negligible sets, there exists an open representative.
\end{thm}

\section{Staircasing for the denoising problem}\label{chapTV:chap_staircasing}
As recalled in our introduction, it has been long observed that the minimizer of this problem has unnatural homogeneous regions referred to as the \textit{staircase} regions. The problem of proving the existence of these constant zones has been tackled in the discrete setting in \cite{Nikolova00} but almost nothing is known in the continuous setting. We will prove in this section that the staircasing phenomenon is unavoidable in the continuous setting in dimension $N\geq2$ even though there is no addition of noise.\\

\noindent Staircasing through the level lines:
\begin{figure}[H]   \begin{minipage}[c]{.49\linewidth}
     \includegraphics[width=\linewidth]{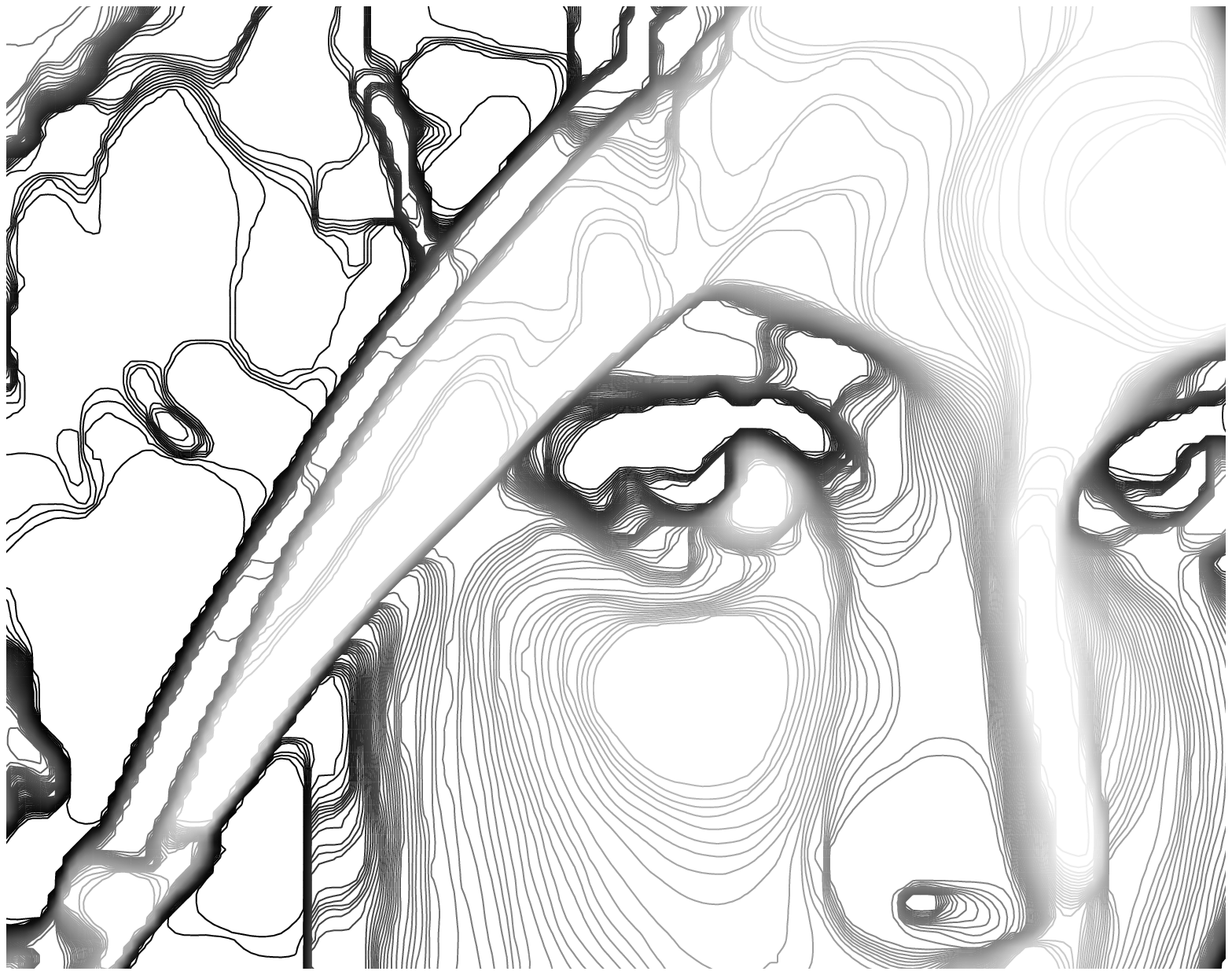}
\vspace{-0.4cm}
    \caption[What will appear in the list of figures.]
{Level lines for the TV-minimizer of the \textit{Lena} image}
  \end{minipage}
\hfill
  \begin{minipage}[c]{.49\linewidth}
     \includegraphics[width=\linewidth]{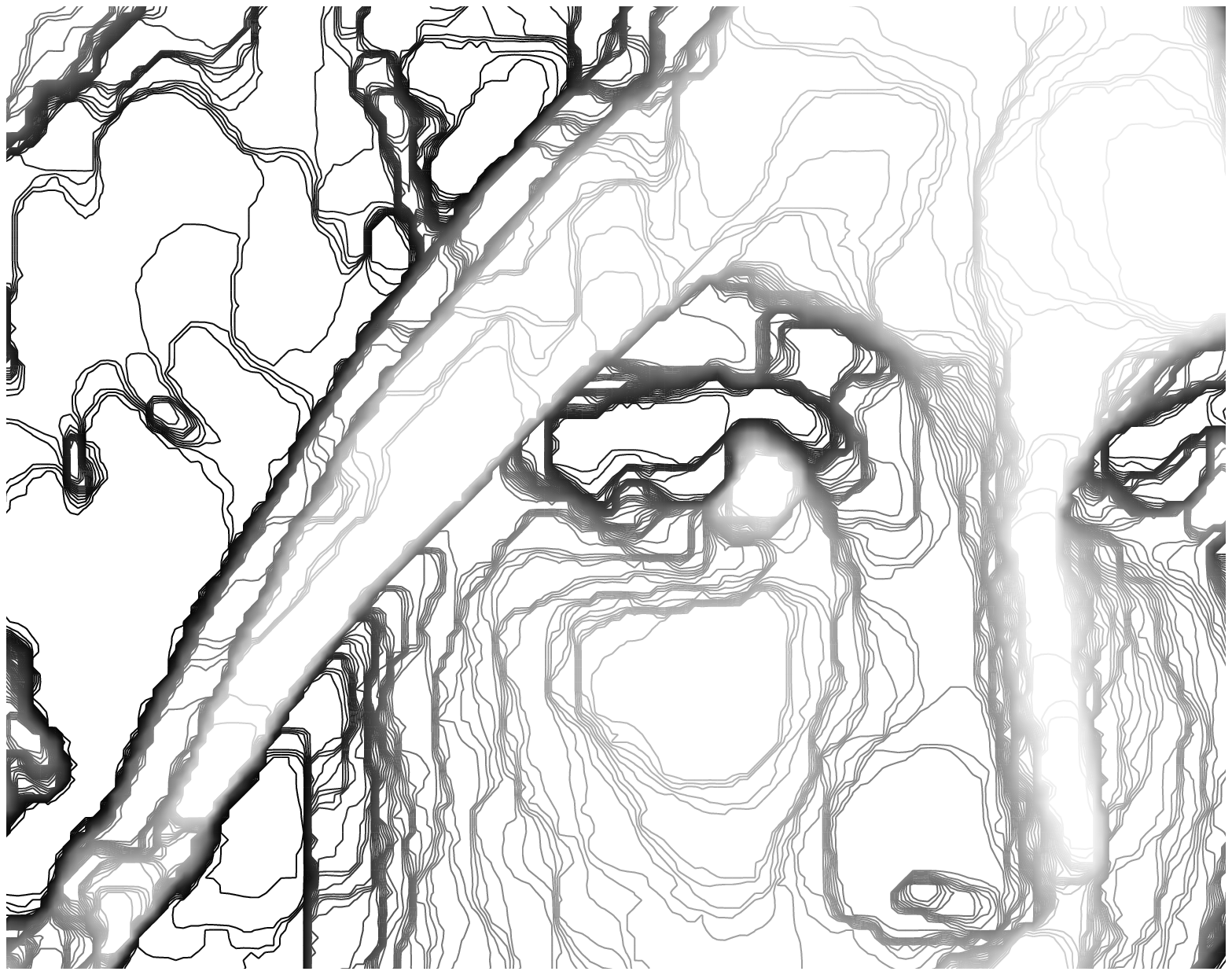}
\vspace{-0.4cm}
    \caption{Level lines for the TV-minimizer of a noisy image}
  \end{minipage} 
\end{figure}

The previous images show that the level lines of the minimizers miss large regions that are therefore constant. The idea of looking at the level sets gives a good intuition but it is also a key point in our analysis to establish results from a theoretical point of view.

\subsection{Staircasing at extrema} 
Let us start by stating one of the main results of this section: 

\begin{thm}\label{chapTV:staircase_bound} Let $g\in L^2(\R^N)$ bounded from above. Then the associated minimizer $u_\lambda$ of (\ref{chapTV:ROF}) 
($\lambda>0$) is also bounded from above, attains its maximum and one has
$$|\{u_\lambda=\max u_\lambda\}|>0.$$ 
In particular $Du_\lambda=0\text{ in } \{u_\lambda=\displaystyle\max u_\lambda\}$. A similar result holds for the minimum when $g$ is bounded from below.
\end{thm}

\begin {proof}
We denote $m_g:=\sup g$. 
Proving that there is a staircase amounts to show that the superlevel sets vanish at some point. Let us therefore consider  the superlevel $E_t^\lambda:=\{u_\lambda>t\}$ for some real $t$. 
By Theorem \ref{chapTV:LevelSetPb}, we know that $E_t^\lambda$ minimizes the following problem
\begin{align*}
\min_E\lambda P(E)+\int_E(t-g(x))dx.
\end{align*}
In particular,
\begin{align}\label{chapTV:eqPer}
\lambda P\left(E_t^\lambda\right)+\int_{E_t^\lambda}(t-g(x))dx\leq 0.
\end{align}
By the isoperimetric inequality (see Theorem \ref{chapTV:Sobo}) and equation (\ref{chapTV:eqPer}), we get 
$${|E_t^\lambda|}^{\frac{N-1}{N}}\leq P\left(E_t^\lambda\right)\leq\frac{1}{\lambda}\int_{E_t^\lambda}(g(x)-t)dx\leq|E_t^\lambda|\frac{m_g-t}{\lambda}.$$
This implies immediately that ${|E_t^\lambda|}=0$ for any $t\geq m_g$. Actually, by a thresholding argument, one can show that $E_t^\lambda=\emptyset$, but this additional information will not be needed here.\\

If for some $t\in\R$, $|E_t^\lambda|\not=0$ then
\begin{align}\label{chapTV:ineg_niv}
{|E_t^\lambda|}\geq{\left(\frac{\lambda}{m_g-t}\right)}^{N}. 
\end{align}
Let us assume that we have $|E_t^\lambda|\not=0$ for any $t_0<t<m_g$.
Then letting $t\to m_g$ in (\ref{chapTV:ineg_niv}) contradicts $|E_t^\lambda|<+\infty$. 
So, if we set $m_u:=\sup u_\lambda$, 
we therefore have $m_u=\sup\{t,|E_t^\lambda|\not=0\}<m_g$.\\

Now, we would like to prove that $|\{u_\la=m_u\}|\not= 0$. This is indeed true since
\begin{align*}
\{u_\la=m_u\}=\bigcap_{n\in\N}E_{m_u-\frac{1}{n}}^\lambda
\end{align*}
and by (\ref{chapTV:ineg_niv}) it follows
\begin{align*}
|\{u_\la=m_u\}|\geq\lim_{n\to\infty}\left|E_{m_u-\frac{1}{n}}^\lambda\right|\geq{\left(\frac{\lambda}{m_g-m_u}\right)}^{N}>0.
\end{align*}

By the coarea formula, we then simply have
\begin{align*}
\int_{\{u_\la=m_u\}}|Du_\lambda|&=\int_{-\infty}^{+\infty}P(E_t^\lambda\cap \{u_\la=m_u\})dt\\
&=\int_{m_u}^{m_g}P(E_t^\lambda)dt\\
&=0.
\end{align*}
This implies that $Du_\lambda=0$ on the staircase set $\{u_\la=m_u\}$ that has positive Lebesgue measure.
\end {proof}

\begin{rmq}\label{chapTV:rmq_stairc}
\begin{trivlist}
\item [$(i)$] In case $g\in L^\infty(\R^N)$ is not assumed to be constant, we actually proved that $u$ is also bounded and that we have
$$\inf_{\R^N} g < \min_{\R^N} u_\lambda \leq \max_{\R^N} u_\lambda < \sup_{\R^N} g.$$
Moreover, inequality (\ref{chapTV:ineg_niv})
gives a lower bound for the staircasing effect namely
\begin{align}\label{chapTV:staircase_quant}
|\{Du_\lambda=0\}|&\geq{|\{u_\lambda=\min_{\R^N} u_\lambda\}\cup\{u_\lambda=\max_{\R^N} u_\lambda\}|}\notag\\
&\geq2{\left(\frac{\lambda}{\sup_{\R^N} g-\inf_{\R^N} g}\right)}^{N}.
\end{align}

\item [$(ii)$] Our proof is rather simple but as was stated in Theorem \ref{chapTV:LevelSetPb}, it is possible to prove that $\{u_\lambda=\max_\O u_\lambda\}$ has an open representative. However, the proof is not direct and relies on the well-known density estimate for quasi-minimizers of the perimeter (see Lemma \ref{chapTV:density} and the proof of Theorem \ref{chapTV:thm_staircase_max}).\\  

\end{trivlist}
\end{rmq}
As we just saw, Theorem \ref{chapTV:staircase_bound} furnishes a way to quantify the staircase effect through the inequality (\ref{chapTV:staircase_quant}). Nonetheless, this bound is not sharp as can be seen for $g$ the characteristic of a convex set (see below). The reason is that we do not take into account creation of flat zones occurring near local extrema. This is the object of the following theorem
\newline
\begin{thm}\label{chapTV:thm_staircase_max}
Let $g\in L^p(\O)$ with $p\in(N,+\infty]$ and $u_\lambda$ with $\lambda>0$ the corresponding minimizer of (\ref{chapTV:ROF}). If $x_0$ is a local maximum of $u_\lambda$ then 
 there exists $C>0$ dpending only on $N$ such that
\begin{align*}
\liminf_{r\to 0 }\frac{|\{u_\la=u_\la^+(x_0)\}\cap B(x_0,r)|}{|B(x_0,r)|}>C.
\end{align*}
In particular, for any radius $\rho>0$,
\begin{align*}
|\{u_\la=u_\la^+(x_0)\}\cap B(x_0,\rho)|>0.
\end{align*}
Moreover, if $x_0\not\in J_{u_\la}$ there exists $\mathcal{N}(x_0)$ a neighborhood of $x_0$ such that
\begin{align*}
D u_\lambda=0 \text{ on } \mathcal{N}(x_0).
\end{align*}
A similiar result holds for a local minimum.\\
\end{thm}

To prove the theorem we shall need the following result:\\
\begin{lem}[Density estimate]\label{chapTV:density}
Consider $E$ a set of finite perimeter that is a minimizer of the perimeter problem (\ref{chapTV:LevelSetPbEq}). Then, denoting $\omega_N$ the volume of the unit ball in $\R^N$, there exists a radius $r_0>0$ 
such that for any point $x\in\O$,
\begin{itemize}
\item[-] if $\forall r>0,\ |E\cap B(x,r)|>0 \text{ then } \forall r<r_0,\ |E\cap B(x,r)|\geq \frac{\omega_{N}r^{N}}{2^N},$
\item[-] if $\forall r>0,\ |B(x,r)\setminus E |>0 \text{ then } \forall r<r_0,\ |B(x,r)\setminus E|\geq \frac{\omega_{N}r^{N}}{2^N}$.
\end{itemize}
In particular, if $x\in\partial^* E$,
\begin{align*}
\forall r<r_0,\ \min{\left(|E\cap B(x,r)|,|B(x,r)\setminus E|\right)} \geq \frac{\omega_{N}r^{N}}{2^N}.
\end{align*}
\end{lem}
One may find a proof of this classical result in \cite[Lemma 5.1]{JalalzaiPhD} and also in \cite{ChamThou,Caffarelli,ChamRegularity}. With the previous lemma in hands we can now turn to the proof of Theorem \ref{chapTV:thm_staircase_max}:\\

\begin{proof}
Without loss of generality, we can consider a local maximum point $x_0$ of $u_\la$ that is to say there is a radius $\rho>0$ such that
\begin{align*}
u_\la\leq u_\la^{+}(x_0) \text{ on } B(x_0,\rho).
\end{align*}
Considering that $x_0\in E_t^\la=\{u_\la>t\}$ for some $t<u_\la^+(x_0)$, one has for any $r>0$,
\begin{align*}
{|E_t^\la\cap B(x_0,r)|}>0,
\end{align*}
and by the Density Lemma \ref{chapTV:density} 
\begin{align*}
\liminf_{r\to 0}\frac{|E_t^\la\cap B(x_0,r)|}{\omega_N r^N}\geq\frac{1}{2^N}.
\end{align*}
Now given that for any $r>0$,
\begin{align*}
{|\{u_\la=u_\la^+(x_0)\}\cap B(x_0,r)|}=&\inf_{\e> 0}{|E^\la_{u_\la^+(x_0)-\e}\cap B(x_0,r)|}
\end{align*}
then by the Lebesgue's density theorem,
\begin{align*}
\liminf_{r\to 0}\frac{|\{u_\la=u_\la^+(x_0)\}\cap B(x_0,r)|}{\omega_N r^N}=&
\inf_{\e>0}\liminf_{r\to 0}\frac{|E^\la_{u_\la^+(x_0)-\e}\cap B(x_0,r)|}{\omega_N r^N}\geq \frac{1}{2^N}.
\end{align*}
So there is some small $r<\rho$ such that
\begin{align*}
|\{u_\la=u_\la^+(x_0)\}\cap B(x_0,\rho)|\geq\frac{{\omega_N r^N}}{2^N}.
\end{align*}
Moreover, if we identify $\{u_\la=u_\la^+(x_0)\}$ with the set of points where it has density one, then it is an open set.
\end{proof}

Putting together these two theorems we get\\
\begin{cor}
Let $g\in L^2(\R^N)\cap L^\infty(\R^N)$ and  $u_\lambda$, $\lambda>0$ be the minimizer of (\ref{chapTV:ROF}) associated to $g$. If $u_\lambda$ is constant on at most two disjoint sets then it has no local extrema other than its global maximum and minimum.
\end{cor}

\subsection{Dependency of the staircasing on \texorpdfstring{$\lambda$}{the regularization parameter}}
\label{chapTV:stairc_evo}

In the previous section, we proved that for fixed $\lambda$ staircasing always occurs and can be quantified by (\ref{chapTV:staircase_bound}). This bound suggests that the Lebesgue measure of the created flat zones is non-decreasing with respect to $\lambda$. This was already observed for the total variation flow in $\R^N$. In \cite[Chapter 4]{Andreu} the authors even prove that the solution $u(t)$ of the total variation flow in $\R^N$ decreases in time, for some norm, with a finite extinction time. 
It is possible to get a similar result for the minimizer $u_\lambda$:
\newline
\begin{prop}\label{chapTV:prop_Bougain}
Let $\O$ be a connected bounded Lipschitz continuous open set in $\R^N$, $g\in  L^N(\O)$ and $u_\lambda$ the minimizer of (\ref{chapTV:ROF}). Then, there exists $\lambda^*=C_\O{\|g\|}_N\geq 0$ with $C_\O$ that only depends on $\O$, such that for any $\lambda\geq\lambda^*$,
\begin{align*}
u_\lambda=
\frac{1}{|\O|}\int_\O g.
\end{align*}
\end{prop}

\begin{proof}
The conclusion follows readily if one can find some $z_\la$ that satisfies the system
\begin{align*}
\begin{cases}
-\la\diverg z_\la+\frac{1}{|\O|}\int_\O g=g &\text{ in } \O,\\
{\|z_\la\|}_\infty\leq 1 &\text{ in } \O,\\
z_\la\cdot\nu_{\O}=0 &\text{ on }\partial\O.\\
\end{cases}
\end{align*}
If $p\geq N$, the function $g-\frac{1}{|\O|}\int_\O g$ is in $L^N(\O)$ and of mean zero. Therefore, the result of Bourgain and Brezis \cite{Bourgain} (see also \cite{DePauw}) asserts that there exists a $z\in L^\infty(\O,\R^N)\cap W^{1,N}(\O,\R^N)$ that solves
\begin{align*}
\begin{cases}
-\diverg z=g-\frac{1}{|\O|}\int_\O g &\text{ in } \O,\\
z\cdot\nu_{\O}=0 &\text{ on } \partial\O.\\
\end{cases}
\end{align*}
Thus a nice candidate is $z_\la=\frac{z}{\la}$, for $\la$ large enough. Actually, $\la^*={\|z\|}_\infty$ that is controlled by ${\|g\|}_N$.
\end{proof}

If the domain is not bounded, the result of Bourgain and Brezis does not apply. However, one has a similar result for a data $g$ that lies in the so-called \textit{Schwartz class} $\mathcal{S}$ that contains those functions whose derivatives are decreasing faster that any polynomial (see \cite{Hormander}):
\newline
\begin{prop}\label{chapTV:prop_Fourier}
If $g$ is an element of $\mathcal{S}(\R^N)$ and $u_\la$ denotes the minimizer of (\ref{chapTV:ROF}) corresponding to $g$ then there exists $\la^*\geq0$ such that for $\la\geq\la^*$
$$u_\la=0 \text{ in } \R^N.$$
\end{prop}

\begin{proof}
As in the previous proof, the assertion follows if one can find some $z_\la$ that solves the following system
\begin{align*}
\begin{cases}
-\la\diverg z_\la=g,\\
{\|z_\la\|}_\infty\leq 1.
\end{cases}
\end{align*}
Let us look for a $z_\la$ that is of the form $z_\la=\frac{\G v}{\la}$ with $v$ that satisfies in $\R^N$
\begin{align}\label{chapTV:laplacien}
-\Delta v = g.
\end{align}
Taking the Fourier transform on both sides, one gets
\begin{align*}
-4\pi^2|\xi|^2\hat{v}(\xi)=\hat{g}(\xi),\ \forall\xi\in\R^N
\end{align*}
hence the following estimate
\begin{align*}
{\|\G v\|}_\infty\leq{\|\widehat{\G v}\|}_1=\frac{1}{4\pi^2}{\left\|{|\xi|^{-1}\hat{g}(\xi)}\right\|}_1<+\infty
\end{align*}
where in the last inequality we used the well-known fact that $\hat{g}\in\mathcal{S}\subset L^p(\R^N)$ for any $p\in[1,\infty]$ (see \cite{Hormander} for instance for further details).
Therefore $z_\la=\frac{\nabla v}{\la}$ satisfies the system above as soon as $\la\geq\la^*:= {\|\G v\|}_\infty$.
\end{proof}

For a general unbounded subdomain of $\R^N$, the previous proof cannot be adapted since it is well known that for a bounded $g$, equation (\ref{chapTV:laplacien}) does not necessarily admit a solution $v$ in $W^{2,\infty}$ (see \cite{Bourgain}).\\ 

In case $N=2$, it is possible to prove Proposition \ref{chapTV:prop_Bougain} without having to use the difficult result of Bourgain and Brezis: 
\newline
\begin{prop}
Let $\O\subset\R^2$ a be connected open set that is bounded with a Lipschitz continuous boundary, $g\in L^2(\O)$ and $u_\la$ the minimizer of (\ref{chapTV:ROF}) associated to $g$. Then there exists $\la^*={C_\O}{{\|g\|}_2}$, with $C_\O$ that only depends on $\O$, such that for $\la\geq\la^*$
$$u_\la=\frac{1}{|\O|}\int_\O g.$$
\end{prop}

\begin{proof}
Given $u\in\BV\cap L^2(\O)$, in \cite[Lemma 2.4]{Andreu}, one can find the following characterization of $p=-\diverg(z)\in\partial TV(u)$: 
\begin{align*}
\int_\O|Du|\leq\int_{\O}(u-\vphi)p+\int_\O z\cdot\G\vphi,\ \forall \vphi\in W^{1,1}(\O)\cap L^2(\O).
\end{align*}
Therefore, denoting $\langle u_\la\rangle$ the average of $u_\la$ and setting $u=u_\la$, $\vphi=\langle u_\la \rangle$ and $p=\frac{1}{\la}(g-u_\la)$ one has
\begin{align*}
\int_\O|Du_\la|\leq\frac{1}{\la}\int_\O(g-u_\la)\left(u_\la-\langle u_\la\rangle\right).
\end{align*}
Then, applying the Poincaré inequality (see Theorem \ref{chapTV:Sobo}) and Cauchy-Schwarz and using $\langle u_\la\rangle=\langle g\rangle$, the average of $g$, one obtains the estimate
\begin{align*}
C{\|u_\la-\langle g\rangle\|}_2\leq\frac{1}{\la}{\|u_\la-\langle g\rangle\|}_2{\|u_\la-g\|}_2,
\end{align*}
where $C$ is the constant that appears in the Poincaré inequality.
Now remarking that, by minimality of $u_\la$,
\begin{align*}
{\|u_\la- g\|}_2\leq{\|g\|}_2
\end{align*}
concludes the proof. Moreover, we get $\la^*=\frac{{\|g\|}_2}{C}$.
\end{proof}


By the Sobolev inequality (see Theorem \ref{chapTV:Sobo}), for which the optimal constant is known, it is readily checked that we also gave an alternative proof for Proposition \ref{chapTV:prop_Fourier} in case $N=2$:
\newline
\begin{prop}
Let $g\in L^2(\R^2)$ and $u_\la$ be the minimizer of (\ref{chapTV:ROF}) associated to $g$ then there exists $\la^*={\left({2\pi^{\frac{1}{2}}}\right)}^{-1}{\|g\|}_2$ such that for $\la\geq\la^*$
$$u_\la=0 \text{ in } \R^N.$$
\end{prop}

\begin{rmq}
Reasoning as we did we get an explicit $\la^*$ that is optimal as can be seen for $g=\chi_{\mathbb{D}}$ the characteristic of the unit disc since we know in this case that
\begin{align*}
u_\la=(1-2\la)^+\chi_\D.
\end{align*}
\end{rmq}

Now that we got rid of the case when $\lambda$ is large, let us see through some examples how the staircase regions behave for reasonable values of the regularization parameter. In the rest of this section, we assume that $g$ is the characteristic function of a set $C$ which means we are now interested in the minimizers of
\begin{align*}
\E_\lambda(u)=\lambda\int_\O|Du|+\frac{1}{2}{\|u-\chi_C\|}^2_2.
\end{align*}
We are going to distinguish two different cases for $C$:

\subsubsection{The characteristic of a bounded convex set in $\R^2$} Then it is well known (see \cite{Allard08,Alter})
that for $\la\leq\la^*$ and for any $t\geq0$ the superlevel $E_t^\la=\{u_\lambda>t\}$ of the solution $u_\lambda$ is given by
\begin{align*}
E_t^\la=
\begin{cases}
C_{\lambda/(1-t)} \text{ if }t\leq 1-\lambda/R^*,\\
\emptyset\text{ otherwise.}
\end{cases}
\end{align*}
where for any $R>0$, $C_R$ is the opening of $C$ defined by
\begin{align*}
C_R=\bigcup_{B(x,R)\subset C} B(x,R),
\end{align*}
and $R^*$ is the inverse of the so-called \textit{Cheeger constant} defined by the value of $R$ that solves
\begin{align*}
\frac{P(C_R)}{|C_R|}=\frac{1}{R}.
\end{align*}
Therefore, the staircase set 
\begin{align*}
\{u_\lambda=\max u_\la\}=C_{R^*}
\end{align*}
is the so-called \textit{Cheeger set} and is independent of $\lambda\leq\la^*$. We would get similar results if $C$ were the union of spaced convex sets (see \cite{Andreu} for the expression of $u_\lambda$ in this case).

\subsubsection{The characteristic of two touching squares in $\R^2$} Here $C=[0,1]\times[0,-1]\cup[-1,0]\times[0,1]$ is the union of two unit squares that only touch on a vertex.
In \cite{Allard09}, Allard gives a full description of the superlevels of the solution $u_\lambda$, hence $u_\lambda$ itself. The level sets $E_t^\la=\{u_\la>t\}$ are of five kinds namely
\begin{align*}
E_t^\la\in\left\{\emptyset,\ \R^2,\ F_{r,s},\ G_{r,s},\ H_{r}\ /\ \la=\frac{1}{r}+\frac{1}{s},r,s\in\R^+\right\}
\end{align*}
where the last three sets are formally defined in \cite{Allard09}. They are depicted in the following figures as the interior of the domain bounded by the red curve:
\begin{figure}[H]
\label{chapTV:levels_squares}
\centering
   \begin{minipage}[c]{.32\linewidth}
\input{./Allard1_vect.txt}
    \caption
{$F_{r,s}$.}
\end{minipage}
   \begin{minipage}[c]{.32\linewidth}
\input{./Allard2_vect.txt}
    \caption
{$G_{r,s}$.}
\end{minipage}
   \begin{minipage}[c]{.32\linewidth}
\input{./Allard3_vect.txt}
    \caption
{$H_{r}$.}
\end{minipage}
\end{figure}

\noindent In the following figure Allard summed up the different possibilities:
\begin{figure}[H]
\label{chapTV:stack_levels}
   \begin{minipage}[c]{.46\linewidth}
\input{./Allard4_vect.txt}
    \caption
{Level sets $E_t^\la$.}
  \end{minipage}
\hfill
  \begin{minipage}[c]{.46\linewidth}
     \includegraphics[width=.76\hsize]{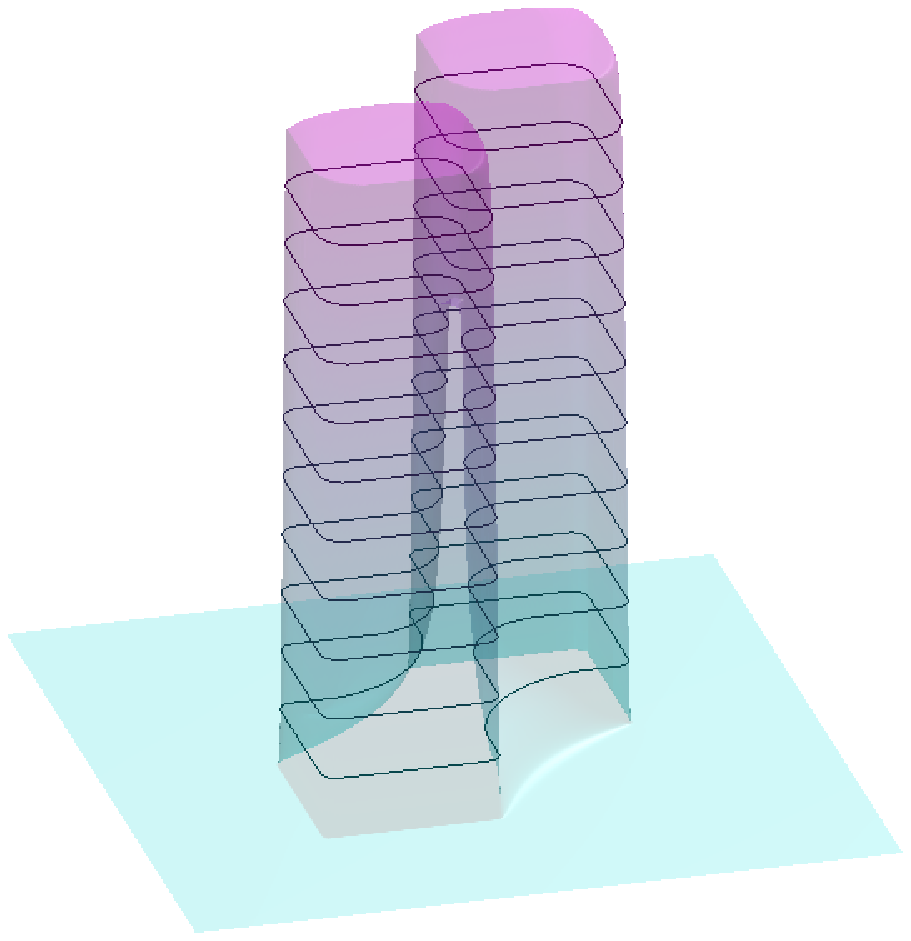}
    \caption
{$u_\la$ and its level lines.}
  \end{minipage} 
\end{figure}


Now consider that $\la\in(0,\la^*)$. As seen previously, the staircase set \linebreak $S_\la:=\{u_{\la}=\max u_\la\}$ is exactly the smallest superlevel set that is not empty this is to say
\begin{align*}
S_\la\in\{F_{q(s),s},\ G_{q(s),s},\ H_{q(s)}\}
\end{align*}
where the function $q(s)$ is non-decreasing in $s$. 
Let us focus on two values of the regularization parameter: $\la_1=\frac{7}{32}$ and $\la_2=\frac{11}{64}$ and let $s_i,\ i\in\{1,2\}$ be the unique value such that $\la_i=\frac{1}{s_i}+\frac{1}{q(s_i)}$. Then it is readily seen that
\begin{align*}
S_{\la_2}=G_{q(s_2),s_2}\not\subset F_{q(s_1),s_1}=S_{\la_1}
\end{align*}
even though $\la_2<\la_1$.

This example shows that in general the staircase zones $(S_\la)_{\la\geq0}$ do not form a monotone sequence.\\

Such a phenomenon cannot occur given a radial function $g$ and this is what we are going to prove in the following section. In short, this is due to the fact that the solutions ${(u_\la)}_\la$ of the radial problem form a semi-group. This was already established in the one-dimensional setting \cite{Briani} and in case $g=\chi_C$ the characteristic function of a convex set $C$ \cite{Andreu,Alter}.

\section{Denoising problem for radial data}
\label{chapTV:denoising_radial}
Unless otherwise stated, in this section $\O$ will denote the ball $B(0,R)\subset\R^N$ with $N\geq2$ and we consider a radial $g\in L^2(\O)$. It is easily seen that the minimizer $u_\la$ of (\ref{chapTV:ROF}) is itself radial. Indeed one could argue that for any rotation $\mathcal{R}$, $u_\la(\mathcal{R}x)$ is also a minimizer. Since it is unique $u_\la(\mathcal{R}x)=u_\la(x)$ for any $x\in\R^N$. We denote $\tilde{v}$ the function defined by $\tilde{v}(|x|)=v(x)$ for any $x\in \R^N$. 
Then ${u}_\la$ minimizes
\begin{align}\label{chapTV:ROF_rad}
\min_{u\in\BV} \int_0^R\left(\la|\tilde{u}'(r)|+\frac{1}{2}{(\tilde{u}(r)-\tilde{g}(r))}^2\right)r^{N-1}dr.
\end{align}
Thus, proceeding as in the one-dimensional case, either ${u}_\la$ is constant or
\begin{align*}
{u}_\la(x)=g(x)+\sgn{(\tilde{u}_\la'\big(|x|)\big)}\frac{(N-1)\la}{|x|}
\end{align*}
for any $x\in\R^N\setminus\{0\}$.
Now, introducing the dual variable $\ti{z}$
and reasoning as in \cite{ChamAlgo}, we can derive a dual formulation for the minimization problem (\ref{chapTV:ROF_rad}), that is
\begin{align*}
\inf_{\substack{\ti{z}\in W^{1,2}(0,R),\\
 \ti{z}(R)=0,\ |\ti{z}|\leq\la}} \int_0^R \left(\ti{z}'(r)+\frac{N-1}{r}\tilde{z}(r)+g(r)\right)^2 r^{N-1}dr.
\end{align*}
Then, if we set $z(x)=\tilde{z}(x)\frac{x}{|x|}$ for any $x\in \R^N\setminus\{0\}$,
\begin{align*}
(\diverg(z))(x)=\G(\tilde{z}(|x|))\cdot \frac{x}{|x|}+\tilde{z}(|x|)\diverg\left(\frac{x}{|x|}\right)=\tilde{z}'(|x|)+\tilde{z}(|x|)\frac{N-1}{|x|}
\end{align*}
which gives the dual formulation 
\begin{align}\label{chapTV:ROF_rad_dual}
\inf_{z\in{A_\la}}\E(z):=\int_{B(0,R)} {\left(\diverg(z)+g\right)}^2,
\end{align}
where
\begin{align*}
A_\la=\left\{z(x)=\tilde{z}(x)\frac{x}{|x|} \ /\ \forall x\in\R^N\setminus\{0\}, \ti{z}\in W^{1,2}(0,R),\ \ti{z}(R)=0,\ |\ti{z}|\leq\la\right\}
\end{align*}
is the set of admissible $z$. Henceforth, we denote $z_\la$ a minimizer of problem (\ref{chapTV:ROF_rad_dual}). Since it is radial, we will sometimes write $z_\la$ instead of $\ti{z}_\la$.  Note also that $z_\la$ is actually continuous by the Sobolev embedding theorem.\\

\begin{rmq}
We decided to proceed this way to justify rigorously that one can pick a radial vectorfield $z_\la$ in the Euler-Lagrange equation (\ref{chapTV:ROF_EL}).
\end{rmq}

\subsection{The solutions form a semi-group}
Let us start with the following lemma:
\begin{lem}\label{lem_semi-grp}
Let $g\in L^p(\O)$ with $p\in [N,+\infty]$ then for any $\la>0$, 
\begin{align*}
|z_\la(x)|\leq C|x|^{1-\frac{N}{p}}{\|g-u_\la\|}_{L^p(B(0,|x|))}
\end{align*}
for some positive $C$ hence in particular $z_\la(0)=0$.
\end{lem}

\begin{proof}
For a.e. $r=|x|$ with $x\in\R^N$
\begin{align*}
\diverg(z_\la)(x)=\ti{z}_\la'(r)+\frac{N-1}{r}\ti{z}_\la(r)=\frac{1}{r^{N-1}}(\ti{z}_\la(r)r^{N-1})'.
\end{align*}
Hence integrating with respect to $r$ and using $\diverg(z_\la)\leq |g-u_\la|$, it follows
\begin{align*}
\ti{z}_\la(r)r^{N-1}\leq \int_{B(0,r)}|g-u_\la|\leq C {\|g-u_\la\|}_{L^p(B(0,r))}r^{N\left(1-\frac{1}{p}\right)}
\end{align*}
for some positive real $C$.
\end{proof}

The following proposition is the key result in the study of the radial problem:
\begin{prop}[Comparison result]\label{chapTV:comp_pp}
Let $g\in L^2(\O)$, $\mu>\la\geq0$ and consider respectively $z_\la$ and $z_\mu$ the corresponding minimizers of $\E$ then one has
\begin{align*}
z_\la-\la\geq z_\mu-\mu.
\end{align*}
\end{prop}
\begin{proof}
On the one hand,
\begin{align*}
\E\left(\Big({z}_\mu+(\la-\mu)\frac{x}{|x|}\Big)\vee {z}_\la\right)=&\int_{\{z_\la>z_\mu+\la-\mu\}}{\left(\diverg{(z_\la)}+g\right)}^2dx\\
+&\int_{\{z_\la<z_\mu+\la-\mu\}}{\left(\diverg\Big(z_\mu+(\la-\mu)\frac{x}{|x|}\Big)+g\right)}^2dx,\\
\E\left(\Big({z}_\la+(\mu-\la)\frac{x}{|x|}\Big)\wedge {z}_\mu\right)=&\int_{\{z_\la>z_\mu+\la-\mu\}}{\left(\diverg(z_\mu)+g\right)}^2dx\\
+&\int_{\{z_\la<z_\mu+\la-\mu\}}{\left(\diverg\Big(z_\la+(\mu-\la)\frac{x}{|x|}\Big)+g\right)}^2dx,
\end{align*}
thus by minimality of $z_\mu$ and $z_\la$
\begin{align*}
\int_{\{z_\la<z_\mu+\la-\mu\}}\left({(\diverg(z_\la)+g)}^2+{(\diverg(z_\mu)+g)}^2\right)dx\\
\leq\int_{\{z_\la<z_\mu+\la-\mu\}}{\left(\diverg\Big(z_\mu+(\la-\mu)\frac{x}{|x|}\Big)+g\right)}^2dx\\
+\int_{\{z_\la<z_\mu+\la-\mu\}}{\left(\diverg\Big(z_\la+(\mu-\la)\frac{x}{|x|}\Big)+g\right)}^2dx,
\end{align*}
and then it follows that
\begin{align}\label{chapTV:rad_N}
\int_{\{z_\la-z_\mu+\mu-\la<0\}} (\mu-\la)\frac{N-1}{|x|}\left(\diverg\Big(z_\la-z_\mu+({\mu-\la})\frac{x}{|x|}\Big)\right)dx\geq0.
\end{align}
But on the other hand, given that $z_\la-z_\mu+\mu-\la>0$ on $\partial\O$ and $z_\la-z_\mu+\mu-\la=0$ on $\partial\{z_\la-z_\mu+\mu-\la<0\}$, one gets by integration for $\e>0$, 
\begin{align*}
\int_{[\O\setminus B(0,\e)]\cap\{z_\la-z_\mu+\mu-\la<0\}} \frac{N-1}{|x|}\left(\diverg\Big(z_\la-z_\mu+({\mu-\la})\frac{x}{|x|}\Big)\right)dx\\
=-\frac{N-1}{\e}\int_{\partial B(0,\e)\cap\{z_\la-z_\mu+\mu-\la<0\}}(z_\la-z_\mu+\mu-\la)(x)dx\\
+\int_{ [\O\setminus B(0,\e)]\cap\{z_\la-z_\mu+\mu-\la<0\}}\frac{N-1}{{|x|}^2}(z_\la-z_\mu+\mu-\la)(x)dx.
\end{align*}
Now one has to distinguish two cases:\\

If $N\geq 3$, then, given that $C:={\|z_\la-z_\mu+\mu-\la\|}_\infty<+\infty$, one has 
\begin{align*}
0\leq-\frac{N-1}{\e}\int_{\partial B(0,\e)\cap\{z_\la-z_\mu+\mu-\la<0\}}(z_\la-z_\mu+\mu-\la)(x)dx\leq 2\pi C\e^{N-2}.
\end{align*}
Thus sending $\e\to0$ and considering identity $(\ref{chapTV:rad_N})$ we obtain
\begin{align*}
\int_{\{z_\la-z_\mu+\mu-\la<0\}}\frac{N-1}{{|x|}^2}(z_\la-z_\mu+\mu-\la)(x)dx \geq 0
\end{align*}
hence
\begin{align*}
z_\la-\la\geq z_\mu-\mu.
\end{align*}
 
If $N=2$ and $\e$ is small then, by Lemma \ref{lem_semi-grp}, 
\begin{align*}
{\partial B(0,\e)\cap\{z_\la-z_\mu+\mu-\la<0\}}=\emptyset
\end{align*}
thus
\begin{align*}
-\frac{N-1}{\e}\int_{\partial B(0,\e)\cap\{z_\la-z_\mu+\mu-\la<0\}}(z_\la-z_\mu+\mu-\la)(x)dx=0
\end{align*}
and the conclusion follows in the same way as above.
\end{proof}

\begin{cor}
For any $\la,\mu\geq0$, one has
\begin{align*}
{\|z_\la-z_\mu\|}_\infty\leq|\la-\mu|.
\end{align*}
\end{cor}

Another very important consequence of Proposition \ref{chapTV:comp_pp} is that solutions ${(u_\la)}_{\la\in\R^+}$ form a semigroup. To prove this, let us denote
\begin{align*}
T_\la(g)=\argmin_{u\in\BV} \la\int_\O|Du|+\frac{1}{2}{\|u-g\|}_2^2
\end{align*}
for any real $\la\geq0$ and any datum $g\in L^2(\O)$. Thus,
\begin{align*}
T_\la={(I+\la\partial TV)}^{-1}	
\end{align*}
is the \textit{resolvent operator} discussed in \cite{Brezis73}.
\begin{prop}\label{chapTV:semi-grp}
If $\O=B(0,R)$, $g\in L^2(\O)$ is radial, $\mu>\la\geq0$ are two regularization parameters then
\begin{align*}
T_\mu(g)=T_{\mu-\la}\circ T_\la(g).
\end{align*}
\end{prop}
\begin{proof}
\textit{Step 1:}
Let us set
\begin{align*}
z_0=\argmin_{z\in A_{\mu-\la}} \int_\O(\diverg(z)+T_\la(g))^2
\end{align*}
and ${z}_\mu'=z_0+z_\la$. We claim that $\diverg(z_\mu')=\diverg(z_\mu)$.\\ 
On the one hand, we know from the previous corollary that $z_\mu-z_\la\in A_{\mu-\la}$ thus by comparison with $z_0$
\begin{align*}
\int_\O(\diverg(z_\mu'-z_\la)+T_\la(g))^2\leq\int_\O(\diverg(z_\mu-z_\la)+T_\la(g))^2.
\end{align*}
and using the Euler-Lagrange equation it follows that
\begin{align*}
\int_\O(\diverg(z_\mu')+g)^2\leq\int_\O(\diverg (z_\mu)+g)^2.
\end{align*}
Whereas, on the other hand, $z_\mu'\in A_\mu$ thus by comparison with $z_\mu$
\begin{align*}
\int_\O(\diverg(z_\mu)+g)^2\leq\int_\O(\diverg (z_\mu')+g)^2
\end{align*}
which proves our claim.

\noindent\textit{Step 2:} Now considering the following Euler-Lagrange equations
\begin{align*}
\begin{cases}
\diverg(z_\mu)=T_\mu(g)-g,\\
\diverg(z_\la)=T_\la(g)-g,\\
\diverg(z_0)=T_{\mu-\la}\circ T_\la(g)-T_\la(g),
\end{cases}
\end{align*}
the conclusion follows readily from the result of Step 1.
\end{proof}

Let us recall that given an open set $\O\subset\R^N$ and an initial condition \linebreak $u(0)=g\in L^2(\O)$ there exists a unique solution  to the gradient flow equation (see \cite{Andreu}):
\begin{align}\label{chapTV:flot_TV}
-\partial_t u(t)\in\partial TV(u(t)) \text{ a.e. in }t\in[0,T].
\end{align}
We are also going to need the following classical result for the resolvent operator of a maximal monotone operator (see \cite[Corollary 4.4]{Brezis73}):
\newline
\begin{prop}\label{chapTV:Euler_sch}
Let $\O$ be an open set of $\R^N$ and $u(t)$ be the solution of (\ref{chapTV:flot_TV}) with an initial condition $u(0)=g\in L^2(\O)$ then
\begin{align*}
\lim_{n\to+\infty} {T_{\frac{t}{n}}}^{n}(u(0))=\lim_{n\to+\infty}{\left(I+\frac{t}{n}\partial TV\right)}^{-n}\big(u(0)\big)=u(t),
\end{align*}
the convergence taking place in $L^2(\O)$.
\end{prop}

With these results in hands we are ready to prove the main result of this section, namely
\begin{thm}\label{chapTV:flow_rad}
Let $\O=B(0,R)$, $g\in L^2(\O)$ be radial and $u(t)$ be the solution of (\ref{chapTV:flot_TV}) with an initial condition $u(0)=g$. Then $u(t)$ is the unique minimizer of
\begin{align}\label{chapTV:ROF_t}
\min_{u\in\BV} t\int_\O|Du|+\frac{1}{2}{\|u-g\|}_2^2.
\end{align}
\end{thm}

\begin{proof}
From Proposition \ref{chapTV:semi-grp} it follows
\begin{align*}
{T_{\frac{t}{n}}}^{n}(u(0))=T_t(u(0))
\end{align*}
and making $n\to+\infty$ one has by Proposition \ref{chapTV:Euler_sch}
\begin{align*}
u(t)=T_t(u(0))
\end{align*}
hence the statement of the theorem.
\end{proof}

\begin{rmq}
This property is not true for general semi-groups $u(t)$. Let us reason by contradiction. If $u(t)$ solves both (\ref{chapTV:flot_TV}) and (\ref{chapTV:ROF_t}) then writing the Euler-Lagrange equation one has for some fixed $t>0$:
\begin{align*}
\begin{cases}
-\partial_t u(t)\in\partial TV(u(t)),\\ 
\frac{g-u(t)}{t}\in\partial TV(u(t)).
\end{cases}
\end{align*}
Placing ourselves on the set $\{\nabla u(t)\not=0\}$, it follows
\begin{align*}
-t\partial_t u = g-u,
\end{align*}
hence
\begin{align*}
\frac{d}{dt}\left(\frac{u}{t}\right)=-\frac{1}{t^2}(u-t\partial_tu)=-\frac{g}{t^2}=\frac{d}{dt}\left(\frac{g}{t}\right).
\end{align*}
Thus, there exists $C(x)$ that does not depend on $t$ such that
\begin{align*}
u(t,x)=g(x)+C(x)t
\end{align*}
whenever $\nabla u(t)\not=0$ in a neighborhood of $x$.
This contradicts the example of Allard where $\O=\R^2$ and $g=\chi_{[0,1]\times[0,-1]\cup[-1,0]\times[0,1]}$. Indeed let us consider the origin $x_0=0$. Then from the analysis of Allard (see section \ref{chapTV:stairc_evo}) one knows that for $0<t<t^*$ there is a neighborhood of $x_0$ that is contained in $\{\nabla u(t)\not=0\}$. Thus we get a contradiction if the minimizer $u_t$ of (\ref{chapTV:ROF}) with parameter $\la=t$ is not affine in $t$ at $x_0$. This is indeed observed by a simple numerical experimentation:
\begin{figure}[H]
\centering
   \begin{minipage}[c]{.47\linewidth}
     \includegraphics[width=\linewidth]{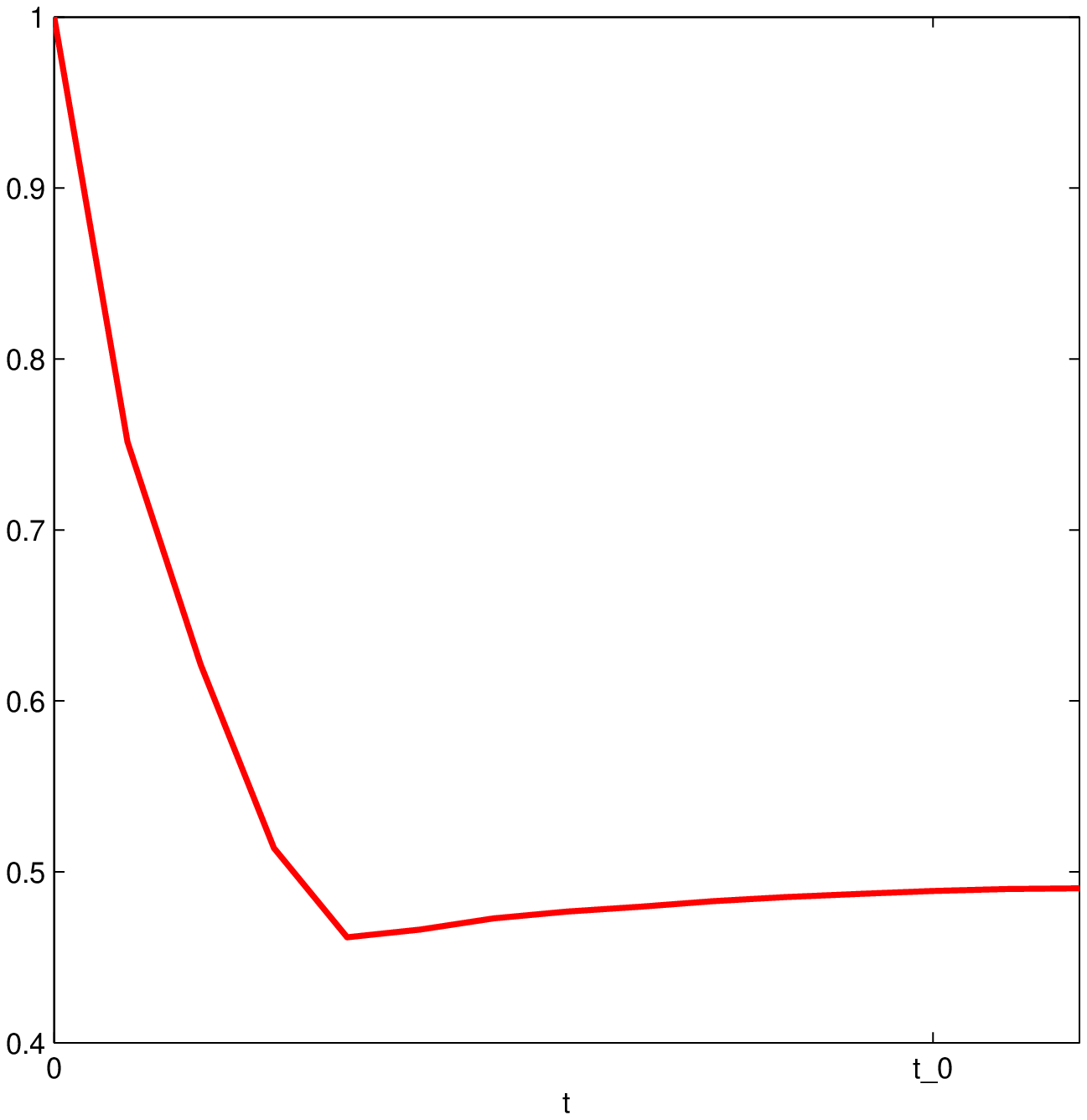}
\vspace{-0.4cm}
    \caption
{Evolution of $u(t,0)$}
\end{minipage}
   \begin{minipage}[c]{.46\linewidth}
     \includegraphics[width=\linewidth]{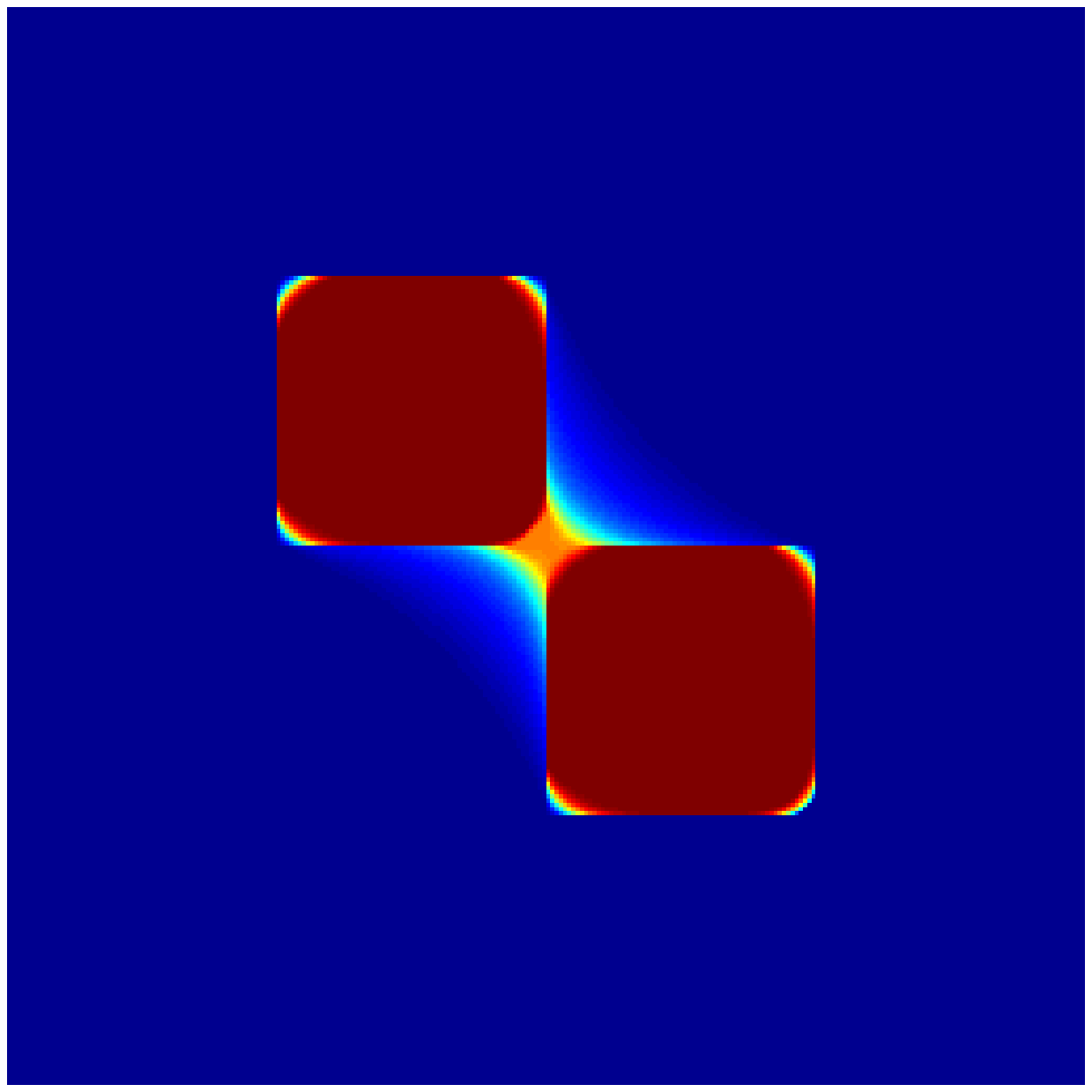}
    \caption
{Solution $u(t)$ for $t=t_0$}
\end{minipage}
\end{figure}
The approach is rigorous since the algorithm used for the minimization of discrete (\ref{chapTV:ROF}) functional is convergent (see \cite{ChamPD} and \cite[Chapter 5]{JalalzaiPhD}) and in \cite{Wang} it is shown that the difference between the continuous (\ref{chapTV:ROF}) model and its finite difference discretization is bounded and tends to zero.

\end{rmq}
\subsection{Staircasing and discontinuities}

From Proposition \ref{chapTV:comp_pp} we can also get some information on the staircase regions. Indeed, whenever condition $|z_\la|\leq\la$ is saturated, we know from the Euler-Lagrange equation that $\G u_\la\not=0$. This way, we can get an inclusion principle for the staircasing namely:
\newline
\begin{thm}
For any reals $\mu>\la\geq 0$, one has 
$\{|z_\la|<\la\}\subset\{|z_\mu|<\mu\}.$\\
\end{thm}
This is not true in general as seen through the example of Allard (see Section \ref{chapTV:stairc_evo}).
\newline
\begin{rmq}Theorem \ref{chapTV:flow_rad} is a key step in \cite{Briani} to prove that in the one-dimensional case, staircasing occurs almost everywhere for a noisy 1D signal. By noisy we mean that the original signal was perturbed by a Wiener process. We cannot expect to get from our analysis such a result simply because a radial function that underwent an addition of noise is not radial anymore. We could however extend their result to the case of a radial ``noise'' but this does not seem really interesting. However, it would tell us again that staircasing is an important phenomenon for minimizers of perturbed signals.
\end{rmq}

As for the discontinuities one can actually refine the results of \cite{JalalzaiJump} thanks to the following
\begin{thm}
Let $\O=B(0,R)$ and $g\in L^N(\O)$ be radial. Consider also $\mu>\la>0$ and $u_\la$, $u_\mu$ the corresponding minimizers of (\ref{chapTV:ROF}). Then one has
\begin{align*}
J_{u_\mu}\subset J_{u_\la}.
\end{align*}
If in addition, $g\in L^N(\O)\cap\BV$ then given $\mu>\la\geq0$
\begin{align*}
J_{u_\mu}\subset J_{u_\la}\subset J_g.
\end{align*}
\end{thm}
\begin{proof}
Once we know Theorem \ref{chapTV:flow_rad} the result is a straightforward consequence of \cite[Theorem 4.1]{ChamJump} and \cite[Theorem 2]{JalalzaiFlow} which establishes a similar inclusion principle for the $TV$ flow.
\end{proof}

\section{Conclusion and perspective}

In this paper, we examined some fine properties of the total variation minimization problem. In particular, we established that the staircasing phenomenon always occurs for the continuous ROF problem. We refined this result in the radial case by proving that the staircase zones are non-decreasing with the regularization parameter. The argument is based on the relation between the ROF problem and the total variation flow, which we prove to hold for radial data. In particular, using known results for the flow we also get that the discontinuities form a monotone sequence given that the datum is radial.\\

Most of the results our rely heavily on the connection with the perimeter problem via the coarea formula and it would interesting to understand how they can be adapted to take into account linear perturbations of the data (convolution but also Radon or Fourier transforms).\\

An interesting extension of our work would be to prove an $N$-dimensional counterpart for the result of the recent paper \cite{Briani}: staircasing occurs almost everywhere for noisy images.\\


\bibliography{biblio}{}
\bibliographystyle{siam}
\end{document}